\documentclass{article}

\usepackage{amssymb,amsmath,color}
\usepackage{amsthm}
\usepackage{url}
\usepackage{tikz}

\definecolor{red}{rgb}{1,0,0}
\definecolor{blue}{rgb}{.2,.2,.8}

\newtheorem{theorem}{Theorem}[section]

\newtheorem{conjecture}{Conjecture}

\theoremstyle{definition}

\begin{document}

\title{The powers of two as sums over partitions}
%\title{New decompositions of the powers of two}
\author{Mircea Merca\\
	%\ and Mircea Merca\footnote{mircea.merca@profinfo.edu.ro}
	% \\ 
	%\small Department of Mathematics,
	%\small University of Craiova\\
	%\small Craiova, 200585 Romania}
	%\date{}
	\footnotesize Department of Mathematics, University of Craiova, 200585 Craiova, Romania\\
	\footnotesize Academy of Romanian Scientists, Ilfov 3, Sector 5, Bucharest, Romania\\
	\footnotesize mircea.merca@profinfo.edu.ro
%	\and Cristian-Silviu Radu
%	\\ 
%	\footnotesize Research Institute for Symbolic Computation\\ 
%	\footnotesize Johannes Kepler University Linz,
%	A-4040 Linz, Austria\\
%	\footnotesize sradu@risc.jku.at
}
\date{}
\maketitle
%\ccom{} \mcom{}

\begin{abstract} 
In this paper, we investigate two methods to express the natural powers of $2$ as sums over integer partitions. First we consider a formula by N. J. Fine that allows us to express a binomial coefficient in terms of multinomial coefficients as a sum over partitions. The second method invokes the central binomial coefficients and the logarithmic differentiation of their generating function. Some experimental results suggest the existence of other methods of decomposing the power of $2$ as sums over partitions.
\\ 
\\
{\bf Keywords:}  binomial coefficients, partitions
\\
\\
{\bf MSC 2010:}   11P81, 11P83, 05A17, 
\end{abstract}

\section{Introduction}

The binomial coefficients are the positive integers that occur as coefficients in the binomial theorem, i.e.,
\begin{equation}\label{EQ2}
(x+y)^n = \sum_{k=0}^n \binom{n}{k} x^k y^{n-k}.
\end{equation}
By this identity, with $x$ and $y$ replaced by $1$, we derive a well known decompostion of $2^n$ in terms of the binomial coefficients:
\begin{equation}\label{EQ3}
2^n = \sum_{k=0}^n \binom{n}{k}.
\end{equation}
Considering the binomial theorem, we can write
$$
\frac{(1+x)^n+(1-x)^n}{2} = \sum_{\substack{k=0\\k\text{ even}}}^n \binom{n}{k} x^{k}
$$
By this identity, with $x$ replaced by $1$, we deduce the following bisection  of equation \eqref{EQ3}:  
\begin{equation}\label{EQ3a}
2^{n-1} = \sum_{\substack{k=0\\k \text{ even}}}^n \binom{n}{k} = \sum_{\substack{k=0\\k \text{ odd}}}^n \binom{n}{k}.
\end{equation}

Multinomial coefficients are generalizations of binomial coefficients, with a similar combinatorial interpretation. They are the coefficients of terms in the expansion of a power of a multinomial, i.e.,
\begin{equation}\label{EQ4}
(x_1+x_2+\cdots+x_k)^n = \sum_{\substack{\lambda_1+\lambda_2+\cdots+\lambda_k=n\\\lambda_1,\lambda_2,\ldots,\lambda_k\geqslant 0}} \binom{n}{\lambda_1,\lambda_2,\ldots,\lambda_k} x_1^{\lambda_1}x_2^{\lambda_2}\cdots x_{k}^{\lambda_k}.
\end{equation}
As we can see, the right hand side of this relation is a sum over all the compositions of $n$. 
Recall \cite{Andrews98} that a composition of a positive integer $n$ is a way of writing $n$ as a sum of positive integers, i.e.,
\begin{equation}\label{EQ1}
n=\lambda_1+\lambda_2+\cdots+\lambda_k.
\end{equation}
When the order of integers $\lambda_i$ does not matter, the representation \eqref{EQ1} is known as an integer partition  and can be rewritten as
$$n=t_1+2t_2+\cdots+nt_n,$$
where each positive integer $i$ appears $t_i$ times in the partition. The number of parts of this partition is given by
$$t_1+t_2+\cdots+t_n=k.$$

In the case $n = 2$, the multinomial theorem \eqref{EQ4} reduces to that of the binomial theorem.
The multinomial coefficient can be expressed in numerous ways, including as a product of binomial coefficients:
\begin{eqnarray}\label{EQ5}
\binom{\lambda_1+\lambda_2+\cdots+\lambda_k}{\lambda_1,\lambda_2,\ldots,\lambda_k}=\binom{\lambda_1}{\lambda_1} \binom{\lambda_1+\lambda_2}{\lambda_2} \cdots \binom{\lambda_1+\lambda_2+\cdots+\lambda_k}{\lambda_k}.  
\end{eqnarray}

Another connection between binomial coefficients and multinomial coefficients is given by the following formula published by N. J. Fine \cite[Ex. 5, p. 87]{Fine59}: for $n,k>0$,
\begin{equation} \label{EQ6}
\sum_{\substack{t_1+2t_2+\cdots+nt_n=n\\t_1+t_2+\cdots+t_n=k}}\binom{k}{t_1,t_2,\ldots,t_n}=\binom{n-1}{k-1}.
\end{equation}
In fact, the number of integer compositions of $n$ is $2^{n-1}$ and the number with exactly $k$ parts is $\binom{n-1}{k-1}$. If we ``forget" the order of the parts, we turn a composition into a partition and the number of compositions corresponding to a given partition becomes a matter of arrangements whose answer is a multinomial coefficient. So $\binom{n-1}{k-1}$ can be expressed as a sum over integer partitions, which is true for any binomial coefficient. Considering the relation \eqref{EQ3}, \eqref{EQ3a}, \eqref{EQ5} and \eqref{EQ6}, we easily derive the following 
%equivalent forms of \eqref{EQ3}.
identities.

\begin{theorem}\label{T1}
	For $n>0$,
	\begin{enumerate}
		\item $\displaystyle{\sum_{\substack{\lambda_1+\lambda_2+\cdots+\lambda_n=n\\\lambda_1\geqslant \lambda_2\geqslant \cdots \geqslant \lambda_n\geqslant 0}} \binom{\lambda_1}{\lambda_2} \binom{\lambda_2}{\lambda_3} \cdots \binom{\lambda_n}{0} = 2^{n-1}};$
		\item $\displaystyle{\sum_{\substack{\lambda_1+\lambda_2+\cdots+\lambda_n=n\\\lambda_1\geqslant \lambda_2\geqslant \cdots \geqslant \lambda_n\geqslant 0}} \binom{\lambda_1}{\lambda_2} \binom{\lambda_2}{\lambda_3} \cdots \binom{\lambda_n}{0} \frac{n}{\lambda_1} = 2^{n}-1};$
		\item $\displaystyle{\sum_{\substack{\lambda_1+\lambda_2+\cdots+\lambda_n=n\\\lambda_1\geqslant \lambda_2\geqslant \cdots \geqslant \lambda_n\geqslant 0}} \binom{\lambda_1}{\lambda_2} \binom{\lambda_2}{\lambda_3} \cdots \binom{\lambda_n}{0} \frac{\lambda_1}{n+1} = 2^{n-2}}.$
	\end{enumerate}
\end{theorem}

For example, the following are the partitions of $5$:
\begin{equation*}
5,\ 4+1,\ 3+2,\ 3+1+1,\ 2+2+1,\ 2+1+1+1,\ 1+1+1+1+1.
\end{equation*}
According to Theorem \ref{T1}, we have:
\begin{align*}
 2^4 &= \binom{5}{0}+\binom{4}{1}\binom{1}{0}+\binom{3}{2}\binom{2}{0} + \binom{3}{1}\binom{1}{1}\binom{1}{0}  \\ 
& \qquad + \binom{2}{2}\binom {2}{1} \binom{1}{0} +\binom{2}{1}\binom{1}{1}\binom{1}{1}\binom{1}{0} +\binom{1}{1}\binom{1}{1}\binom{1}{1}\binom{1}{1}\binom{1}{0}
\end{align*}
and
\begin{align*}
2^5-1 &=
\binom{5}{0}\frac{5}{5}+\binom{4}{1}\binom{1}{0}\frac{5}{4}+\binom{3}{2}\binom{2}{0}\frac{5}{3} + \binom{3}{1}\binom{1}{1}\binom{1}{0}\frac{5}{3} \\
& \qquad + \binom{2}{2}\binom {2}{1} \binom{1}{0}\frac{5}{2} + \binom{2}{1}\binom{1}{1}\binom{1}{1}\binom{1}{0}\frac{5}{2} +\binom{1}{1}\binom{1}{1}\binom{1}{1}\binom{1}{1}\binom{1}{0} \frac{5}{1}
\end{align*}
and
\begin{align*}
2^3 &=
\binom{5}{0}\frac{5}{6}+\binom{4}{1}\binom{1}{0}\frac{4}{6}+\binom{3}{2}\binom{2}{0}\frac{3}{6} + \binom{3}{1}\binom{1}{1}\binom{1}{0}\frac{3}{6} \\
& \qquad + \binom{2}{2}\binom {2}{1} \binom{1}{0}\frac{2}{6} + \binom{2}{1}\binom{1}{1}\binom{1}{1}\binom{1}{0}\frac{2}{6} +\binom{1}{1}\binom{1}{1}\binom{1}{1}\binom{1}{1}\binom{1}{0} \frac{1}{6}.
\end{align*}

%We have the following equivalent form of this result. 

\begin{theorem}\label{T1a}
	For $n>1$,
	\begin{enumerate}
		\item $\displaystyle{\sum_{\substack{\lambda_1+\lambda_2+\cdots+\lambda_n=n\\\lambda_1\geqslant \lambda_2\geqslant \cdots \geqslant \lambda_n\geqslant 0\\ \lambda_1\text{ odd}}} \binom{\lambda_1}{\lambda_2} \binom{\lambda_2}{\lambda_3} \cdots \binom{\lambda_n}{0} = 2^{n-2}};$
		\item $\displaystyle{\sum_{\substack{\lambda_1+\lambda_2+\cdots+\lambda_n=n\\\lambda_1\geqslant \lambda_2\geqslant \cdots \geqslant \lambda_n\geqslant 0\\ \lambda_1\text{ even}}} \binom{\lambda_1}{\lambda_2} \binom{\lambda_2}{\lambda_3} \cdots \binom{\lambda_n}{0} = 2^{n-2}};$
		\item $\displaystyle{\sum_{\substack{\lambda_1+\lambda_2+\cdots+\lambda_n=n\\\lambda_1\geqslant \lambda_2\geqslant \cdots \geqslant \lambda_n\geqslant 0\\ \lambda_1\text{ odd}}} \binom{\lambda_1}{\lambda_2} \binom{\lambda_2}{\lambda_3} \cdots \binom{\lambda_n}{0} \frac{n}{\lambda_1} = 2^{n-1}};$
		\item $\displaystyle{\sum_{\substack{\lambda_1+\lambda_2+\cdots+\lambda_n=n\\\lambda_1\geqslant \lambda_2\geqslant \cdots \geqslant \lambda_n\geqslant 0\\ \lambda_1\text{ even}}} \binom{\lambda_1}{\lambda_2} \binom{\lambda_2}{\lambda_3} \cdots \binom{\lambda_n}{0} \frac{n}{\lambda_1} = 2^{n-1}-1};$
		\item $\displaystyle{\sum_{\substack{\lambda_1+\lambda_2+\cdots+\lambda_n=n\\\lambda_1\geqslant \lambda_2\geqslant \cdots \geqslant \lambda_n\geqslant 0\\ \lambda_1\text{ odd}}} \binom{\lambda_1}{\lambda_2} \binom{\lambda_2}{\lambda_3} \cdots \binom{\lambda_n}{0} \frac{\lambda_1}{n+1} = 2^{n-3}};$
		\item $\displaystyle{\sum_{\substack{\lambda_1+\lambda_2+\cdots+\lambda_n=n\\\lambda_1\geqslant \lambda_2\geqslant \cdots \geqslant \lambda_n\geqslant 0\\ \lambda_1\text{ even}}} \binom{\lambda_1}{\lambda_2} \binom{\lambda_2}{\lambda_3} \cdots \binom{\lambda_n}{0} \frac{\lambda_1}{n+1} = 2^{n-3}}.$		
	\end{enumerate}
\end{theorem}

The case $n=5$ of this theorem reads as follows:
\begin{align*}
2^3 &=
\binom{5}{0}+\binom{3}{2}\binom {2}{0} + \binom{3}{1}\binom{1}{1}\binom{1}{0} +\binom{1}{1}\binom{1}{1}\binom{1}{1}\binom{1}{1}\binom{1}{0},\\
2^3 &= \binom{4}{1}\binom{1}{0}+\binom{2}{2}\binom{2}{1} \binom{1}{0} 
  +\binom{2}{1}\binom{1}{1}\binom{1}{1}\binom{1}{0},\\
2^4 &=
\binom{5}{0}\frac{5}{5}+\binom{3}{2}\binom {2}{0}\frac{5}{3} + \binom{3}{1}\binom{1}{1}\binom{1}{0}\frac{5}{3} +\binom{1}{1}\binom{1}{1}\binom{1}{1}\binom{1}{1}\binom{1}{0}\frac{5}{1},\\
2^4-1 &= \binom{4}{1}\binom{1}{0}\frac{5}{4}+\binom{2}{2}\binom{2}{1} \binom{1}{0} \frac{5}{2}
+\binom{2}{1}\binom{1}{1}\binom{1}{1}\binom{1}{0}\frac{5}{2},\\
2^2 &=
\binom{5}{0}\frac{5}{6}+\binom{3}{2}\binom {2}{0}\frac{3}{6} + \binom{3}{1}\binom{1}{1}\binom{1}{0}\frac{3}{6} +\binom{1}{1}\binom{1}{1}\binom{1}{1}\binom{1}{1}\binom{1}{0}\frac{1}{6},\\
2^2 &= \binom{4}{1}\binom{1}{0}\frac{4}{6}+\binom{2}{2}\binom{2}{1} \binom{1}{0} \frac{2}{6}
+\binom{2}{1}\binom{1}{1}\binom{1}{1}\binom{1}{0}\frac{2}{6}.
\end{align*}

In this paper, motivated by Theorems \ref{T1} and \ref{T1a}, we shall provide other decompositions of the powers of $2$ in terms of the binomial coefficients as sums over integer partitions.

\begin{theorem} \label{T2}
	For $n>0$,
	\begin{enumerate}
		\item $\displaystyle{\sum_{\substack{\lambda_1+\lambda_2+\cdots+\lambda_n=n\\\lambda_1\geqslant \lambda_2\geqslant \cdots \geqslant \lambda_n\geqslant 0}} \binom{\lambda_1}{\lambda_2} \binom{\lambda_2}{\lambda_3} \cdots \binom{\lambda_n}{0} \frac{n}{\lambda_1} \frac{1^{\lambda_1}1^{\lambda_2}3^{\lambda_3}\cdots(2n-3)^{\lambda_n}}
			{1^{\lambda_1}2^{\lambda_2}\cdots n^{\lambda_n}}= 2^{n-1}};$
		\item $\displaystyle{\sum_{\substack{\lambda_1+\lambda_2+\cdots+\lambda_n=n\\\lambda_1\geqslant \lambda_2\geqslant \cdots \geqslant \lambda_n\geqslant 0}} (-1)^{1+\lambda_1}\binom{\lambda_1}{\lambda_2} \binom{\lambda_2}{\lambda_3} \cdots \binom{\lambda_n}{0} \frac{n}{\lambda_1} \frac{1^{\lambda_1}3^{\lambda_2}\cdots(2n-1)^{\lambda_n}}
			{1^{\lambda_1}2^{\lambda_2}\cdots n^{\lambda_n}}= 2^{n-1}}.$
	\end{enumerate}
\end{theorem}

The case $n=5$ of this theorem reads as follows:
\begin{align*}
2^4 &= \binom{5}{0}\frac{5}{5}\frac{1^5}{1^5}
+\binom{4}{1}\binom{1}{0}\frac{5}{4}\frac{1^41^1}{1^42^1}
+\binom{3}{2}\binom{2}{0}\frac{5}{3}\frac{1^31^2}{1^32^2}
+\binom{3}{1}\binom{1}{1}\binom{1}{0}\frac{5}{3}\frac{1^31^13^1}{1^32^13^1}\\
&\qquad+\binom{2}{2}\binom{2}{1}\binom{1}{0}\frac{5}{2}\frac{1^21^23^1}{1^22^23^1}
+\binom{2}{1}\binom{1}{1}\binom{1}{1}\binom{1}{0}\frac{5}{2}\frac{1^21^13^15^1}{1^22^13^14^1}\\
&\qquad+\binom{1}{1}\binom{1}{1}\binom{1}{1}\binom{1}{1}\binom{1}{0}\frac{5}{1}\frac{1^11^13^15^17^1}{1^12^13^14^15^1} 
\end{align*}
and
\begin{align*}
2^4
& =  \binom{5}{0}\frac{5}{5}\frac{1^5}{1^5}
-\binom{4}{1}\binom{1}{0}\frac{5}{4}\frac{1^43^1}{1^42^1}
+\binom{3}{2}\binom{2}{0}\frac{5}{3}\frac{1^33^2}{1^32^2}
+\binom{3}{1}\binom{1}{1}\binom{1}{0}\frac{5}{3}\frac{1^33^15^1}{1^32^13^1}\\
&\qquad -\binom{2}{2}\binom{2}{1}\binom{1}{0}\frac{5}{2}\frac{1^23^25^1}{1^22^23^1}
-\binom{2}{1}\binom{1}{1}\binom{1}{1}\binom{1}{0}\frac{5}{2}\frac{1^23^15^17^1}{1^22^13^14^1}\\
&\qquad +\binom{1}{1}\binom{1}{1}\binom{1}{1}\binom{1}{1}\binom{1}{0}\frac{5}{1}\frac{1^13^15^17^19^1}{1^12^13^14^15^1}.
\end{align*}

The rest of this paper is organized as follows. We will first prove Theorem \ref{T1} in
Section \ref{S2}. In Section \ref{S3}, we will provide the proof of Theorem \ref{T2}. In Section \ref{S4}, we will conjecture a further decomposition of $2^n$ as sum over all the integer partitions of $n$.

\section{Proof of Theorems \ref{T1} and \ref{T1a}}
\label{S2}
\allowdisplaybreaks{
Taking into account \eqref{EQ3} and \eqref{EQ6}, we can write
\begin{align}
2^{n-1} & = \sum_{k=0}^{n-1} \binom{n-1}{k} = \sum_{k=1}^n \binom{n-1}{k-1}
= \sum_{k=1}^n \sum_{\substack{t_1+2t_2+\cdots+nt_n=n\\t_1+t_2+\cdots+t_n=k}}\binom{k}{t_1,t_2,\ldots,t_n}\nonumber \\
& = \sum_{t_1+2t_2+\cdots+nt_n=n}\binom{t_1+t_2+\cdots+t_n}{t_1,t_2,\ldots,t_n}\label{EQ2.1}
\end{align}
and
\begin{align}
2^n & = 1 + \sum_{k=1}^n \binom{n}{k} = 1+\sum_{k=1}^n \frac{n}{k} \binom{n-1}{k-1}\nonumber \\
& = 1+\sum_{k=1}^n \sum_{\substack{t_1+2t_2+\cdots+nt_n=n\\t_1+t_2+\cdots+t_n=k}}\binom{k}{t_1,t_2,\ldots,t_n}\frac{n}{k}
\nonumber \\
& =1+ \sum_{t_1+2t_2+\cdots+nt_n=n}\binom{t_1+t_2+\cdots+t_n}{t_1,t_2,\ldots,t_n} \frac{n}{t_1+t_2+\cdots+t_n}\label{EQ2.2}
\end{align}
and
\begin{align}
(n+1)2^{n-2} & = 2^{n-1}+(n-1)2^{n-2} \nonumber \\
& = \sum_{k=1}^{n} \binom{n-1}{k-1} + (n-1) \sum_{k=1}^{n-1} \binom{n-2}{k-1} \nonumber \\
%& = \sum_{k=1}^{n} \binom{n-1}{k-1} + (n-1) \sum_{k=1}^{n-1} \frac{k}{n-1} \frac{n-1}{k}\binom{n-2}{k-1}\nonumber \\
& = \sum_{k=1}^{n} \binom{n-1}{k-1} + \sum_{k=0}^{n-1} k\binom{n-1}{k}\nonumber \\
& = \sum_{k=1}^{n} \binom{n-1}{k-1} + \sum_{k=1}^n (k-1) \binom{n-1}{k-1} \nonumber \\
& = \sum_{k=1}^n \sum_{\substack{t_1+2t_2+\cdots+nt_n=n\\t_1+t_2+\cdots+t_n=k}}\binom{k}{t_1,t_2,\ldots,t_n}k
\nonumber \\
& = \sum_{t_1+2t_2+\cdots+nt_n=n}\binom{t_1+t_2+\cdots+t_n}{t_1,t_2,\ldots,t_n}(t_1+t_2+\cdots+t_n).
\label{EQ2.3}
\end{align}
We see that these decompositions of the powers of $2$ are sums over all the partitions of $n$.
Another useful way to think of a partition is with a Ferrers diagram. Each integer in the partition is represented by a row of dots, and the rows are ordered from longest on the top to shortest at the bottom.  For example, the partition $5+4+2+2$ would be represented by 
$$
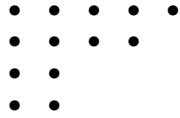
\begin{array}{ccccc}
\bullet&\bullet&\bullet&\bullet&\bullet\\
\bullet&\bullet&\bullet&\bullet\\
\bullet&\bullet\\
\bullet&\bullet
\end{array}
$$
The conjugate of a partition is the one corresponding to the Ferrers diagram produced by flipping the diagram for the original partition across the main diagonal, thus turning rows into columns and vice versa. For the diagram above, the conjugate is
$$
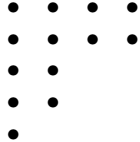
\begin{array}{ccccc}
\bullet&\bullet&\bullet&\bullet\\
\bullet&\bullet&\bullet&\bullet\\
\bullet&\bullet\\
\bullet&\bullet\\
\bullet
\end{array}
$$
with corresponding partition $4+4+2+2+1$.
The action of conjugation takes every partition of one type into a partition of the other: the conjugate of a partition into $k$ parts is a partition with largest part $k$ and vice versa. This establishes a $1–1$ correspondence between partitions into $k$ parts and partitions with largest part $k$. Let $\lambda=(\lambda_1,\lambda_2,\ldots,\lambda_k)$ be the conjugate partition of the partition $t_1+2t_2+\cdots+nt_n=n$. 
It is clear that $\lambda_i = t_i+t_{i+1}+\cdots+t_n$. In this way, considering \eqref{EQ5}, \eqref{EQ2.1}, \eqref{EQ2.2} and \eqref{EQ2.3}, we deduce that
	\begin{align*}
	2^{n-1} & = \sum_{\substack{\lambda_1+\lambda_2+\cdots+\lambda_n=n\\\lambda_1\geqslant \lambda_2\geqslant \cdots \geqslant \lambda_n\geqslant 0}} \binom{\lambda_1}{\lambda_1-\lambda_2,\lambda_2-\lambda_3,\ldots,\lambda_{n-1}-\lambda_n,\lambda_n}\\
	& = \sum_{\substack{\lambda_1+\lambda_2+\cdots+\lambda_n=n\\\lambda_1\geqslant \lambda_2\geqslant \cdots \geqslant \lambda_n\geqslant 0}} \binom{\lambda_n+(\lambda_{n-1}-\lambda_n)+\cdots+(\lambda_1-\lambda_2)}{\lambda_n,\lambda_{n-1}-\lambda_{n},\ldots,\lambda_{1}-\lambda_2}\\
	& = \sum_{\substack{\lambda_1+\lambda_2+\cdots+\lambda_n=n\\\lambda_1\geqslant \lambda_2\geqslant \cdots \geqslant \lambda_n\geqslant 0}} 
	\binom{\lambda_n}{\lambda_n} \binom{\lambda_{n-1}}{\lambda_{n-1}-\lambda_n} \cdots \binom{\lambda_1}{\lambda_1-\lambda_2}
	\end{align*}
	and
	\begin{align*}
	2^{n} & = 1+ \sum_{\substack{\lambda_1+\lambda_2+\cdots+\lambda_n=n\\\lambda_1\geqslant \lambda_2\geqslant \cdots \geqslant \lambda_n\geqslant 0}} 
	\binom{\lambda_n}{\lambda_n} \binom{\lambda_{n-1}}{\lambda_{n-1}-\lambda_n} \cdots \binom{\lambda_1}{\lambda_1-\lambda_2} \frac{n}{\lambda_1}
	\end{align*}
	and
	\begin{align*}
	2^{n-2} & = \sum_{\substack{\lambda_1+\lambda_2+\cdots+\lambda_n=n\\\lambda_1\geqslant \lambda_2\geqslant \cdots \geqslant \lambda_n\geqslant 0}} 
	\binom{\lambda_n}{\lambda_n} \binom{\lambda_{n-1}}{\lambda_{n-1}-\lambda_n} \cdots \binom{\lambda_1}{\lambda_1-\lambda_2} \frac{\lambda_1}{n+1}.
	\end{align*}	
	These conclude the proof of Theorem \ref{T1}. The proof of Theorem \ref{T1a} is quite similar to the proof of Theorem \ref{T1}, so we omit the details.}

\section{Proof of Theorem \ref{T2}}
\label{S3}

In order to prove this theorem, we consider some known generating functions involving the central binomial coefficients $\binom{2n}{n}$:
$$\sum\limits_{n=0}^\infty \binom{2n}{n} z^n = \frac{1}{\sqrt{1-4z}},
\qquad
\sum\limits_{n=0}^\infty n \binom{2n}{n} z^n = \frac{2z}{\sqrt{(1-4z)^3}}$$ 
and 
$$\sum\limits_{n=0}^\infty \frac{1}{2n-1} \binom{2n}{n} z^n = -\sqrt{1-4z}.$$
We can write
\allowdisplaybreaks{
	\begin{align*}
	& \frac{d}{dz} \ln\left( 1+\sum_{n=1}^\infty \binom{2n}{n} z^n \right) 
	= \left( \sum\limits_{n=1}^\infty n \binom{2n}{n} z^{n-1} \right) 
	\left( 1+\sum\limits_{n=1}^\infty \binom{2n}{n} z^n \right) ^{-1}  \\
	& \qquad\qquad = \frac{2} {\sqrt{(1-4z)^3}}\cdot \sqrt{1-4z} = \frac{2}{1-4z} 
	= \sum_{n=1}^\infty 2^{2n-1} z^{n-1}  
	\end{align*}}
and
\allowdisplaybreaks{
	\begin{align*}
	& \frac{d}{dz} \ln\left( 1+\sum_{n=1}^\infty \frac{(-1)^{n-1}}{2n-1} \binom{2n}{n} z^n \right) \nonumber \\
	& \qquad\qquad= \left( \sum\limits_{n=1}^\infty \frac{(-1)^{n-1}n}{2n-1} \binom{2n}{n} z^{n-1} \right) 
	\left( 1+\sum\limits_{n=1}^\infty \frac{(-1)^{n-1}}{2n-1} \binom{2n}{n} z^n \right) ^{-1}  \\
	& \qquad\qquad=\left( 2\sum_{n=0}^\infty (-1)^n \binom{2n}{n} z^n \right) 
	\left( 1+\sum\limits_{n=1}^\infty \frac{(-1)^{n-1}}{2n-1} \binom{2n}{n} z^n \right) ^{-1}  \\
	& \qquad\qquad = \frac{2}{\sqrt{1+4z}} \cdot \frac{1}{\sqrt{1+4z}} = \frac{2}{1+4z} 
	= \sum_{n=1}^\infty (-1)^{n-1} 2^{2n-1} z^{n-1}.
	\end{align*}}
On the other hand, considering the logarithmic series
$$
\ln (1+z) = \sum_{n=1}^\infty  \frac{(-1)^{n-1}}{n} z^n,\qquad |z|<1,
$$
we obtain
\allowdisplaybreaks{
	\begin{align*}
	& \sum_{n=1}^\infty 2^{2n-1} z^{n-1} = \frac{d}{dz} \ln\left( 1+\sum_{n=1}^\infty \binom{2n}{n} z^n \right) 
	= \frac{d}{dz} \sum_{j=1}^\infty \frac{(-1)^{j-1}}{j} \left( \sum_{n=1}^\infty \binom{2n}{n} z^n \right)^j\\
	& = \frac{d}{dz} \sum_{j=1}^\infty \frac{(-1)^{j-1}}{j} \sum_{n=1}^\infty \left( \sum_{\substack{t_1+t_2+\cdots+t_n=j\\t_1+2t_2+\cdots+nt_n=n}} 
	\binom{t_1+t_2+\cdots+t_n}{t_1,t_2,\ldots,t_n} \prod_{i=1}^n \binom{2i}{i}^{t_i} \right) z^n\\
	& = \frac{d}{dz} \sum_{n=1}^\infty \left( \sum_{t_1+2t_2+\cdots+nt_n=n} \frac{(-1)^{1+t_1+t_2+\cdots+t_n}}{t_1+t_2+\cdots+t_n} \binom{t_1+t_2+\cdots+t_n}{t_1,t_2,\ldots,t_n}\times\right. \\
	& \qquad\qquad \left. \times \prod_{i=1}^n \left( \frac{2(2i-1)}{i}\right) ^{t_i+t_{i+1}+\cdots+t_n} \right)  z^n\\
	& = \frac{d}{dz} \sum_{n=1}^\infty \left( \sum_{\substack{\lambda_1+\lambda_2+\cdots+\lambda_n=n\\\lambda_1\geqslant \lambda_2\geqslant \cdots \geqslant \lambda_n\geqslant 0}} \frac{(-1)^{1+\lambda_1}}{\lambda_1} \binom{\lambda_1}{\lambda_2} \binom{\lambda_2}{\lambda_3} \cdots \binom{\lambda_n}{0} \prod_{i=1}^{n} \left( \frac{2(2i-1)}{i}\right) ^{\lambda_i}  \right) z^n\\
	& = \sum_{n=1}^\infty \left( \sum_{\substack{\lambda_1+\lambda_2+\cdots+\lambda_n=n\\\lambda_1\geqslant \lambda_2\geqslant \cdots \geqslant \lambda_n\geqslant 0}} (-1)^{1+\lambda_1}  \frac{n 2^n}{\lambda_1}
	\binom{\lambda_1}{\lambda_2} \binom{\lambda_2}{\lambda_3} \cdots \binom{\lambda_n}{0}
	\prod_{i=1}^{n} \left( \frac{2i-1}{i}\right) ^{\lambda_i}
	\right)  z^{n-1}
	\end{align*}}
and
\allowdisplaybreaks{
	\begin{align*}
	& \sum_{n=1}^\infty (-1)^{n-1} 2^{2n-1} z^{n-1} = \frac{d}{dz} \ln\left( 1+\sum_{n=1}^\infty \frac{(-1)^{n-1}}{2n-1} \binom{2n}{n} z^n \right) \\
	& = \frac{d}{dz} \sum_{j=1}^\infty \frac{(-1)^{j-1}}{j} \left( \sum_{n=1}^\infty \frac{(-1)^{n-1}}{2n-1} \binom{2n}{n} z^n \right)^j\\
	& = \frac{d}{dz} \sum_{j=1}^\infty \frac{(-1)^{j-1}}{j} \sum_{n=1}^\infty \left( \sum_{\substack{t_1+t_2+\cdots+t_n=j\\t_1+2t_2+\cdots+nt_n=n}} 
	\binom{t_1+t_2+\cdots+t_n}{t_1,t_2,\ldots,t_n} \times\right. \\
	& \qquad\qquad \left. \times \prod_{i=1}^n \left(\frac{(-1)^{i-1}}{2i-1} \binom{2i}{i}\right)^{t_i} \right) z^n\\
	& = \frac{d}{dz} \sum_{n=1}^\infty \left( \sum_{t_1+2t_2+\cdots+nt_n=n} \frac{(-1)^{n-1}}{t_1+t_2+\cdots+t_n} \binom{t_1+t_2+\cdots+t_n}{t_1,t_2,\ldots,t_n} \times \right. \\
	& \qquad\qquad \left. \times \prod_{i=1}^n \left( \frac{2|2i-3|}{i}\right) ^{t_i+t_{i+1}+\cdots+t_n} \right)  z^n\\
	& = \frac{d}{dz} \sum_{n=1}^\infty \left( \sum_{\substack{\lambda_1+\lambda_2+\cdots+\lambda_n=n\\\lambda_1\geqslant \lambda_2\geqslant \cdots \geqslant \lambda_n\geqslant 0}} \frac{(-1)^{n-1}}{\lambda_1} 
	\binom{\lambda_1}{\lambda_2} \binom{\lambda_2}{\lambda_3} \cdots \binom{\lambda_n}{0} 
	\prod_{i=1}^{n} \left( \frac{2|2i-3|}{i}\right) ^{\lambda_i}  \right) z^n\\
	& = \sum_{n=1}^\infty \left( \sum_{\substack{\lambda_1+\lambda_2+\cdots+\lambda_n=n\\\lambda_1\geqslant \lambda_2\geqslant \cdots \geqslant \lambda_n\geqslant 0}} (-1)^{n-1}  \frac{n 2^n}{\lambda_1}
	\binom{\lambda_1}{\lambda_2} \binom{\lambda_2}{\lambda_3} \cdots \binom{\lambda_n}{0}
	\prod_{i=1}^{n} \left( \frac{|2i-3|}{i}\right) ^{\lambda_i}  \right)   z^{n-1}.
	\end{align*}
	The proof follows easily by equating the coefficient of $z^{n-1}$ of these identities.
}

\section{Open problems and concluding remarks}
\label{S4}

As we can see in Theorems \ref{T1} and \ref{T1a}, a connection between binomial coefficients and multinomial coefficients due to N. J. Fine \eqref{EQ6} can be used to rewrite
some well known decompositions of the powers of two as sums over integer partitions. Similar results are obtained in Theorems \ref{T2}, considering the logarithmic differentiation.

The identities of Theorem \ref{T1a} can be seen as bisections of the identities of Theorem \ref{T1}. Inspired by this fact, we tried to approach Theorem \ref{T2} from this point of view. In this way, we remark that there is a substantial amount of numerical evidence to state the following conjecture.

\begin{conjecture} \label{C1}
	For $n>1$,
	\begin{enumerate}
		\item $\displaystyle{\sum_{\substack{\lambda_1+\lambda_2+\cdots+\lambda_n=n\\\lambda_1\geqslant \lambda_2\geqslant \cdots \geqslant \lambda_n\geqslant 0\\\lambda_1\text{ odd}}} \binom{\lambda_1}{\lambda_2} \binom{\lambda_2}{\lambda_3} \cdots \binom{\lambda_n}{0} \frac{n}{\lambda_1} \frac{1^{\lambda_1}1^{\lambda_2}3^{\lambda_3}\cdots(2n-3)^{\lambda_n}}
			{1^{\lambda_1}2^{\lambda_2}\cdots n^{\lambda_n}} > 2^{n-2}};$
		\item $\displaystyle{\sum_{\substack{\lambda_1+\lambda_2+\cdots+\lambda_n=n\\\lambda_1\geqslant \lambda_2\geqslant \cdots \geqslant \lambda_n\geqslant 0\\\lambda_1\text{ even}}} \binom{\lambda_1}{\lambda_2} \binom{\lambda_2}{\lambda_3} \cdots \binom{\lambda_n}{0} \frac{n}{\lambda_1} \frac{1^{\lambda_1}1^{\lambda_2}3^{\lambda_3}\cdots(2n-3)^{\lambda_n}}
			{1^{\lambda_1}2^{\lambda_2}\cdots n^{\lambda_n}} < 2^{n-2}};$
		\item $\displaystyle{\sum_{\substack{\lambda_1+\lambda_2+\cdots+\lambda_n=n\\\lambda_1\geqslant \lambda_2\geqslant \cdots \geqslant \lambda_n\geqslant 0\\\lambda_1\text{ odd}}} \binom{\lambda_1}{\lambda_2} \binom{\lambda_2}{\lambda_3} \cdots \binom{\lambda_n}{0} \frac{n}{\lambda_1} \frac{1^{\lambda_1}3^{\lambda_2}\cdots(2n-1)^{\lambda_n}}
			{1^{\lambda_1}2^{\lambda_2}\cdots n^{\lambda_n}}> 2^{n}};$
		\item $\displaystyle{\sum_{\substack{\lambda_1+\lambda_2+\cdots+\lambda_n=n\\\lambda_1\geqslant \lambda_2\geqslant \cdots \geqslant \lambda_n\geqslant 0\\\lambda_1\text{ even}}} \binom{\lambda_1}{\lambda_2} \binom{\lambda_2}{\lambda_3} \cdots \binom{\lambda_n}{0} \frac{n}{\lambda_1} \frac{1^{\lambda_1}3^{\lambda_2}\cdots(2n-1)^{\lambda_n}}
			{1^{\lambda_1}2^{\lambda_2}\cdots n^{\lambda_n}}> 2^{n-1}}.$		
	\end{enumerate}
\end{conjecture}

The first identity of Theorem \ref{T2} can be considered an analogy of the second identity of Theorem \ref{T1}. In this context, the first identity of the following theorem can be considered an analogy of the first identity of Theorem \ref{T1}.

\begin{theorem} \label{T4}
	For $n>0$,
	\begin{enumerate}
		\item $\displaystyle{\sum_{\substack{\lambda_1+\lambda_2+\cdots+\lambda_n=n\\\lambda_1\geqslant \lambda_2\geqslant \cdots \geqslant \lambda_n\geqslant 0}} \binom{\lambda_1}{\lambda_2} \binom{\lambda_2}{\lambda_3} \cdots \binom{\lambda_n}{0}  \frac{1^{\lambda_1}1^{\lambda_2}3^{\lambda_3}\cdots(2n-3)^{\lambda_n}}
			{1^{\lambda_1}2^{\lambda_2}\cdots n^{\lambda_n}}= \frac{(2n-1)!!}{n!}};$
			%\frac{1}{2^n} \binom{2n}{n}};$
		\item $\displaystyle{\sum_{\substack{\lambda_1+\lambda_2+\cdots+\lambda_n=n\\\lambda_1\geqslant \lambda_2\geqslant \cdots \geqslant \lambda_n\geqslant 0}} (-1)^{1+\lambda_1}\binom{\lambda_1}{\lambda_2} \binom{\lambda_2}{\lambda_3} \cdots \binom{\lambda_n}{0}  \frac{1^{\lambda_1}3^{\lambda_2}\cdots(2n-1)^{\lambda_n}}
			{1^{\lambda_1}2^{\lambda_2}\cdots n^{\lambda_n}}= \frac{|2n-3|!!}{n!}}.$
	\end{enumerate}
\end{theorem}

\begin{proof}
	For $n>0 $, we remark that
	$$ \frac{(2n-1)!!}{n!} = \frac{1}{2^n} \binom{2n}{n}\qquad\text{and}\qquad
	\frac{|2n-3|!!}{n!} = \frac{1}{2^n (2n-1)} \binom{2n}{n}.$$
	Equating the coefficient of $z^n$ in
	$$
	\left( \sum\limits_{n=0}^\infty \binom{2n}{n} z^n \right) 
	\left( \sum\limits_{n=0}^\infty \frac{-1}{2n-1} \binom{2n}{n} z^n \right) 
	=\frac{1}{\sqrt{1-4z}} \cdot \sqrt{1-4z} = 1
	$$
	we obtain
	$$
	\sum_{k=0}^n \frac{-1}{2k-1} \binom{2k}{k} \binom{2n-2k}{n-k} = \delta_{0,n},
	$$
	where $\delta_{i,j}$ is the Kronecker delta function.
	Rewriting this identity as
	$$
	\sum_{k=0}^n (-1)^k \frac{(-1)^{k-1}}{2^k (2k-1)}\binom{2k}{k}  \frac{1}{2^{n-k}}\binom{2n-2k}{n-k} = \delta_{0,n},
	$$
	we see that our theorem is the case
	$$a_n = \frac{(-1)^{n-1}}{2^n (2n-1)} \binom{2n}{n} \qquad\text{and}\qquad b_n= \frac{1}{2^{n}}\binom{2n}{n}$$
	of \cite[Theorem 1]{Merca}.  
\end{proof}

We experimentally discover the following analogy of the third identity of Theorem \ref{T1}.

\begin{conjecture} \label{C}
	For $n>0$,
	\begin{enumerate}
		\item $\displaystyle{\sum_{\substack{\lambda_1+\lambda_2+\cdots+\lambda_n=n\\\lambda_1\geqslant \lambda_2\geqslant \cdots \geqslant \lambda_n\geqslant 0}} \binom{\lambda_1}{\lambda_2} \binom{\lambda_2}{\lambda_3} \cdots \binom{\lambda_n}{0} \frac{\lambda_1}{n+1} \frac{1^{\lambda_1}1^{\lambda_2}3^{\lambda_3}\cdots(2n-3)^{\lambda_n}}
			{1^{\lambda_1}2^{\lambda_2}\cdots n^{\lambda_n}}= \frac{(2n)!!-(2n-1)!!}{(n+1)!};}$
		\item $\displaystyle{\sum_{\substack{\lambda_1+\lambda_2+\cdots+\lambda_n=n\\\lambda_1\geqslant \lambda_2\geqslant \cdots \geqslant \lambda_n\geqslant 0}} (-1)^{1+\lambda_1}\binom{\lambda_1}{\lambda_2} \binom{\lambda_2}{\lambda_3} \cdots \binom{\lambda_n}{0} \frac{\lambda_1}{n+1} \frac{1^{\lambda_1}3^{\lambda_2}\cdots(2n-1)^{\lambda_n}}
			{1^{\lambda_1}2^{\lambda_2}\cdots n^{\lambda_n}}= -\frac{(2n-3)!!}{(n+1)!}}.$
	\end{enumerate}
\end{conjecture}

Considering Theorem \ref{T4} and assuming Conjecture \ref{C}, we obtain the following sums over partitions.

\begin{conjecture} \label{C2}
	For $n>1$,
	\begin{enumerate}
		\item $\displaystyle{\sum_{\substack{\lambda_1+\lambda_2+\cdots+\lambda_n=n\\\lambda_1\geqslant \lambda_2\geqslant \cdots \geqslant \lambda_n\geqslant 0}} (1+\lambda_1)
			\binom{\lambda_1}{\lambda_2} \binom{\lambda_2}{\lambda_3} \cdots \binom{\lambda_n}{0}  \frac{1^{\lambda_1}1^{\lambda_2}3^{\lambda_3}\cdots(2n-3)^{\lambda_n}}
			{1^{\lambda_1}2^{\lambda_2}\cdots n^{\lambda_n}}= 2^{n};}$
		\item $\displaystyle{\sum_{\substack{\lambda_1+\lambda_2+\cdots+\lambda_n=n\\\lambda_1\geqslant \lambda_2\geqslant \cdots \geqslant \lambda_n\geqslant 0}} (-1)^{1+\lambda_1}(1+\lambda_1)\binom{\lambda_1}{\lambda_2} \binom{\lambda_2}{\lambda_3} \cdots \binom{\lambda_n}{0}  \frac{1^{\lambda_1}3^{\lambda_2}\cdots(2n-1)^{\lambda_n}}
			{1^{\lambda_1}2^{\lambda_2}\cdots n^{\lambda_n}}= 0}.$
	\end{enumerate}
\end{conjecture}

This conjecture makes us believe that there is another way to decompose $2$'s powers as sums over integer partitions.

%Finally, we want to thank the referees for their helpful comments.

\bigskip

%\noindent\textit{Department of Mathematics,
%University of Craiova, 200585 Craiova, Romania\\
%mircea.merca@profinfo.edu.ro}

\end{document}